\documentclass[preprint,12pt]{elsarticle}
\usepackage{amssymb}
\usepackage{pifont}
\usepackage{amsfonts}
\usepackage{amsmath}
\usepackage{graphicx}
\usepackage{latexsym,amsthm}
\usepackage{hyperref}
\usepackage{combelow}
\usepackage{setspace}
\usepackage[left=1.3cm,top=1.75cm,right=1.3cm]{geometry}

\makeatletter
\def\ps@pprintTitle{%
  \let\@oddhead\@empty
  \let\@evenhead\@empty
  \def\@oddfoot{\reset@font\hfil\thepage\hfil}
  \let\@evenfoot\@oddfoot
}
\makeatother
\allowdisplaybreaks

\newtheorem{lemma}{Lemma}

\newtheorem{remark}{Remark}

\newtheorem{theorem}{Theorem}
\numberwithin{equation}{section}

\begin{document}

%
%
%
%
%
%
%
%
\begin{frontmatter}
\title{Approximation properties of Chlodowsky variant of $(p,q)$ Bernstein-Stancu-Schurer operators}
\author[label1,label2]{Vishnu Narayan Mishra}
\ead{vishnunarayanmishra@gmail.com,
vishnu\_narayanmishra@yahoo.co.in}
\address[label1]{Department of Applied Mathematics \& Humanities,
Sardar Vallabhbhai National Institute of Technology, Ichchhanath Mahadev Dumas Road, Surat -395 007 (Gujarat), India}
\address[label2]{L. 1627 Awadh Puri Colony Beniganj,
Phase -III, Opp. Industrial Training Institute, Ayodhya Main Road, Faizabad-224 001,
(Uttar Pradesh), India } 
\author[label3, label*]{M. Mursaleen}
\ead{mursaleenm@gmail.com}
\address[label3]{Department of Mathematics, Aligarh Muslim University, Aligarh-202002, India}
\fntext[label*]{Corresponding author}
\author[label1]{Shikha Pandey}
\ead{sp1486@gmail.com}
\begin{abstract}
In the present paper, we introduce the Chlodowsky variant of $(p,q)$ Bernstein-Stancu-Schurer operators which is a generalization of $(p,q)$ Bernstein-Stancu-Schurer operators. We have also discussed its Korovkin type approximation properties and  rate of convergence.
\end{abstract}
\end{frontmatter}
\section{Introduction and preliminaries}
In 1912, S.N. Bernstein \cite{snbern} introduced the following sequence of operators $B_n : C[0, 1] \rightarrow C[0, 1]$ defined for any $n \in \mathbb{N}$ and for any function $f \in C[0, 1],$
\begin{equation}
B_n(f;x)=\sum_{k=0}^n \left(\begin{array}{c} n \\ k \end{array} \right) x^k (1-x)^{n-k} f\left(\dfrac{k}{n}\right), ~~~~ x\in[0,1].
\end{equation}
Later various generalizations of these operators were discovered. It has been proved as a powerful tool for numerical analysis, computer aided geometric design and solutions of differential equations. In last two decades, the applications of $q$-calculus has played an important role in the area of approximation theory, number theory and theoretical physics. In 1987, Lupa\cb s \cite{lupas} and in 1997, Phillips \cite{phillip1} introduced a sequence of Bernstein polynomials based on $q$-integers and investigated its approximation properties. Several researchers obtained various other generalizations of operators based on $q$-calculus. For any function  $f \in C[0, 1]$  the $q$-form of Bernstein operator is described by Lupa\cb s \cite{lupas} as
\begin{equation}
L_{n,q}(f;x)=\sum_{k=0}^n \dfrac{\left[\begin{array}{c} n \\ k \end{array} \right]_q q^{k(k-1)/2}x^k (1-x)^{n-k}}{\prod_{j=0}^{n-1}(1-x+q^jx)}f\left(\dfrac{[k]_q}{[n]_q}\right)~~~~ x\in[0,1].
\end{equation}
In 1932, Chlodowsky \cite{Ibikli} presented a generalization of Bernstein polynomials on an unbounded set, known as Bernstein-Chlodowsky polynomials given by,
\begin{equation}
B_n(f,x)=\sum_{k=0}^n f(\dfrac{k}{n}b_n)\left(\begin{array}{c} n \\ k \end{array} \right) \left(\dfrac{x}{b_n}\right)^k \left(1-\dfrac{x}{b_n}\right)^{n-k}, ~~~~~ 0\leq x\leq b_n,
\end{equation}
where $b_n$ is  an increasing sequence of positive terms with the
properties $b_n \rightarrow \infty$ and $\frac{b_n}{n} \rightarrow 0$ as $n \rightarrow \infty.$\\
In 2008, Karsli and Gupta \cite{qChlodowsky} expressed the $q$-analogue of Bernstein-Chlodowsky polynomials by
\begin{equation}
C_n(f;q;x)=\sum_{k=0}^n \left[\begin{array}{c} n \\ k \end{array} \right]_q \left(\dfrac{x}{b_n}\right)^k \prod_{s=0}^{n-k-1} (1-q^s \dfrac{x}{b_n})f\left(\dfrac{[k]_q}{[n]_q}b_n\right), ~~~~~ 0\leq x\leq b_n,
\end{equation}
where $b_n$ is  an increasing sequence of positive values, with the
properties $b_n \rightarrow \infty$ and $\frac{b_n}{[n]_q} \rightarrow 0$ as $n \rightarrow \infty$.\vspace{.5 cm} \\
Recently, Mursaleen et al. \cite{bernpq,bernpqstancu,bernpqschurer} proposed and analysed approximation properties for $(p,q)$ analogue of Bernstein operators , Bernstein-Stancu operators and Bernstein-Schurer operators. Besides this, we
also refer to some recent related work on this topic: e.g. (\cite{vnm}, \cite{opt1}, \cite{acta}, \cite{opt2}, \cite{saha}).
Now, we recall certain notations of $(p,q)$-calculus.\\
For $0<q < p <1$, the $(p,q)$ integer $[n]_{p,q}$ is described by
\begin{equation*}
[n]_{p,q} := \dfrac{p^n-q^n}{p-q}. 
\end{equation*}
$(p,q)$ factorial is expressed as
\begin{equation*}
[n]_{p,q}!=[n]_{p,q} [n-1]_{p,q} [n-2]_{p,q} \ldots 1.
\end{equation*}
$(p,q)$ binomial coefficient is expressed as
\begin{equation*}
\left[\begin{array}{c} n \\ k \end{array} \right]_{p,q} := \dfrac{[n]_{p,q}!}{[k]_{p,q}![n-k]_{p,q}!}.
\end{equation*}
$(p,q)$ binomial expansion as
\begin{equation*}
(ax+by)_{p,q}^n:=\sum_{k=0}^n\left[\begin{array}{c} n \\ k \end{array} \right]_{p,q}a^{n-k}b^k x^{n-k}y^k.
\end{equation*}
$$(x+y)_{p,q}^n := (x+y) (px+qy) (p^2x+q^2y)\ldots (p^{n-1}x+q^{n-1}y).$$
In 2015, Mursleen et al \cite{bernpqstancuRevised}, investigated $(p,q)$ form of Bernstein-Stancu operator, which is given by
\begin{equation}
S_{n}^{(\alpha,\beta)}(f;x,p,q)=\frac{1}{p^\frac{n(n-1)}{2}}\sum_{k=0}^{n} \left[\begin{array}{c} n \\ k \end{array} \right]_{p,q}p^\frac{k(k-1)}{2} x^k \prod_{s=0}^{n+m-k-1} (p^s-q^sx) f \left(\dfrac{p^{n-k}[k]_{p,q}+\alpha}{[n]_{p,q}+\beta}\right),
\end{equation}
where $\alpha, \beta$  are non-negative integers and $f \in C[0, 1]$, $x \in [0,1]$ and $0 \leq \alpha \leq \beta$.\vspace{.5 cm}\\ For the first few moments, we have the following lemma:
\begin{lemma} \label{L1}(See \cite{bernpqstancuRevised}) For the Operators $S_{n}^{(\alpha,\beta)},$  we have
\begin{enumerate}
\item $S_{n}^{(\alpha,\beta)}(1;x,p,q)=1$,
\item $S_{n}^{(\alpha,\beta)}(t;x,p,q)=\dfrac{[n]_{p,q}x+\alpha}{[n]_{p,q}+\beta}$,
\item $S_{n}^{(\alpha,\beta)}(t^2;x,p,q)=\dfrac{1}{([n]_{p,q}+\beta)^2}(q[n]_{p,q}[n-1]_{p,q}x^2+[n]_{p,q}(2\alpha + p^{n-1})x+\alpha^2).$
\end{enumerate}
\end{lemma}
\section{Construction of the operators}
Considering the revised form of $(p,q)$ analogue of Bernstein operators (see \cite{bernRevised}), we construct the Chlodowsky variant of $(p,q)$ Bernstein-Stancu-Schurer operators as
\begin{equation} \label{1}
C_{n,m}^{(\alpha,\beta)}(f;x,p,q)=\frac{1}{p^\frac{n(n-1)}{2}}\sum_{k=0}^{n+m} \left[\begin{array}{c} n+m \\ k \end{array} \right]_{p,q} p^\frac{k(k-1)}{2}\left(\dfrac{x}{b_n}\right)^k \prod_{s=0}^{n+m-k-1} (p^s-q^s\dfrac{x}{b_n}) f \left(\dfrac{p^{n-k}[k]_{p,q}+\alpha}{[n]_{p,q}+\beta}b_n\right),
\end{equation}
where $n\in\mathbb{N}$, $m,\alpha,\beta \in\mathbb{N}_0$ with $0 \leq \alpha \leq \beta$, $0 \leq x \leq b_n$, $0<q < p<1.$ Evidently, $C_{n,m}^{(\alpha,\beta)}$
is a linear and positive operator. Consider the case, if $p,q \rightarrow 1$ and $m= 0$ in (\ref{1}), then it will reduce to the Stancu-Chlodowsky polynomials \cite{stancuChlodo}.\vspace{.5 cm} \\
To begin with, we obtain the following lemma, which will be used throughout the paper.
\begin{lemma}\label{L2} Let $C_{n,m}^{(\alpha,\beta)}(f;x,p,q)$ be given by (\ref{1}). The first few moments of the operators are
\begin{enumerate}
\item[(i)] $C_{n,m}^{(\alpha,\beta)}(1;x,p,q)=1,$
\item[(ii)] $C_{n,m}^{(\alpha,\beta)}(t;x,p,q)=\dfrac{[n+m]_{p,q}x+\alpha b_n}{[n]_{p,q}+\beta},$
\item[(iii)] $C_{n,m}^{(\alpha,\beta)}(t^2;x,p,q)=\dfrac{1}{([n]_{p,q}+\beta)^2}(q[n+m]_{p,q}[n+m-1]_{p,q}x^2+[n+m]_{p,q}(2\alpha+p^{n-1})b_n x+\alpha^2 b_n^2),$
\item[(iv)] $C_{n,m}^{(\alpha,\beta)}((t-x);x,p,q)=\left(\dfrac{[n+m]_{p,q}}{[n]_{p,q}+\beta}-1\right)x+\dfrac{\alpha b_n}{[n]_{p,q}+\beta},$
\item[(v)] \begin{eqnarray*}
 C_{n,m}^{(\alpha,\beta)}((t-x)^2;x,p,q) &=&\Biggl(1-2\frac{[n+m]_{p,q}}{([n]_{p,q}+\beta)}+\frac{q[n+m]_{p,q}[n+m-1]_{p,q}}{([n]_{p,q}+\beta)^2} \Biggl)x^2\\ & &+
   \Biggl( \frac{(2\alpha+p^{n-1})[n+m]_{p,q}}{([n]_{p,q}+\beta)}-2\alpha\Biggl) \dfrac{b_n}{([n]_{p,q}+\beta)} x +\dfrac{\alpha^2 b_n^2}{([n]_{p,q}+\beta)^2}.
\end{eqnarray*}

\end{enumerate} 
\end{lemma}
\begin{proof} (i)
$$C_{n,m}^{(\alpha,\beta)}(1;x,p,q)=\frac{1}{p^\frac{n(n-1)}{2}}\sum_{k=0}^{n+m} \left[\begin{array}{c} n+m \\ k \end{array} \right]_{p,q} p^\frac{k(k-1)}{2}\left(\dfrac{x}{b_n}\right)^k \prod_{s=0}^{n+m-k-1} (p^s-q^s\dfrac{x}{b_n})=1,$$
(ii)\begin{eqnarray*}
&&C_{n,m}^{(\alpha,\beta)}(t;x,p,q)=\frac{1}{p^\frac{n(n-1)}{2}}\sum_{k=0}^{n+m} \left[\begin{array}{c} n+m \\ k \end{array} \right]_{p,q} p^\frac{k(k-1)}{2}\left(\dfrac{x}{b_n}\right)^k \prod_{s=0}^{(n+m-k-1)} (p^s-q^s\dfrac{x}{b_n}) \left(\dfrac{p^{n-k}[k]_{p,q}+\alpha}{[n]_{p,q}+\beta}b_n\right)\\&& \hspace{.2 cm} =
\frac{[n+m]_{p,q}}{p^\frac{n(n-3)}{2}([n]_{p,q}+\beta)}\sum_{k=0}^{n+m-1} \left[\begin{array}{c} n+m-1 \\ k \end{array} \right]_{p,q} p^\frac{k(k+1)}{2}\left(\dfrac{x}{b_n}\right)^{k+1} \prod_{s=0}^{(n+m-k-2)} (p^s-q^s\dfrac{x}{b_n}) \left(\dfrac{b_n}{p^{k+1}}\right)\\ &&\hspace{.4 cm}+ \frac{1}{p^\frac{n(n-1)}{2}}\sum_{k=0}^{n+m} \left[\begin{array}{c} n+m \\ k \end{array} \right]_{p,q} p^\frac{k(k-1)}{2}\left(\dfrac{x}{b_n}\right)^k \prod_{s=0}^{(n+m-k-1)} (p^s-q^s\dfrac{x}{b_n}) \left(\dfrac{\alpha}{[n]_{p,q}+\beta}b_n\right)\\&& \hspace{.2 cm}=
\frac{[n+m]_{p,q}}{([n]_{p,q}+\beta)}x \frac{1}{p^\frac{(n-1)(n-2)}{2}}\sum_{k=0}^{n+m-1} \left[\begin{array}{c} n+m-1 \\ k \end{array} \right]_{p,q} p^\frac{k(k-1)}{2}\left(\dfrac{x}{b_n}\right)^k \prod_{s=0}^{n+m-k-2} (p^s-q^s\dfrac{x}{b_n}) + \frac{\alpha b_n }{([n]_{p,q}+\beta)}\\&& \hspace{.2 cm}=
\frac{[n+m]_{p,q}}{([n]_{p,q}+\beta)}x+\frac{\alpha b_n }{([n]_{p,q}+\beta)},
\end{eqnarray*}
(iii)\begin{eqnarray*}
&& C_{n,m}^{(\alpha,\beta)}(t^2;x,p,q)=\frac{1}{p^\frac{n(n-1)}{2}}\sum_{k=0}^{n+m} \left[\begin{array}{c} n+m \\ k \end{array} \right]_{p,q} p^\frac{k(k-1)}{2}\left(\dfrac{x}{b_n}\right)^k \prod_{s=0}^{(n+m-k-1)} (p^s-q^s\dfrac{x}{b_n}) \left(\dfrac{p^{n-k}[k]_{p,q}+\alpha}{[n]_{p,q}+\beta}b_n\right)^2\\&& \hspace{.2 cm} =
\frac{1}{p^\frac{n(n-1)}{2}}\frac{1}{([n]_{p,q}+\beta)^2}\Biggl[ p^{2n}\sum_{k=0}^{n+m} \left[\begin{array}{c} n+m \\ k \end{array} \right]_{p,q} p^\frac{k(k-1)}{2}\left(\dfrac{x}{b_n}\right)^k \prod_{s=0}^{(n+m-k-1)} (p^s-q^s\dfrac{x}{b_n}) \dfrac{b_n^2[k]_{p,q}^2}{p^{2k}}\\&& \hspace{.2 cm} + 
 2\alpha p^n \sum_{k=0}^{n+m} \left[\begin{array}{c} n+m \\ k \end{array} \right]_{p,q} p^\frac{k(k-1)}{2}\left(\dfrac{x}{b_n}\right)^k \prod_{s=0}^{(n+m-k-1)} (p^s-q^s\dfrac{x}{b_n}) \dfrac{b_n^2[k]_{p,q}}{p^{k}}\\&& \hspace{.2 cm} + 
 \alpha^2 \sum_{k=0}^{n+m} \left[\begin{array}{c} n+m \\ k \end{array} \right]_{p,q} p^\frac{k(k-1)}{2}\left(\dfrac{x}{b_n}\right)^k \prod_{s=0}^{(n+m-k-1)} (p^s-q^s\dfrac{x}{b_n})b_n^2\Biggl] \\&& \hspace{.2 cm}=
\frac{1}{([n]_{p,q}+\beta)^2}\Biggl[\frac{p^{2n}}{p^\frac{n(n-1)}{2}}[n+m]_{p,q}b_n^2 \sum_{k=0}^{n+m-1} \left[\begin{array}{c} n+m-1 \\ k \end{array} \right]_{p,q} p^\frac{k(k+1)}{2}\left(\dfrac{x}{b_n}\right)^{k+1} \prod_{s=0}^{(n+m-k-2)} (p^s-q^s\dfrac{x}{b_n}) \dfrac{[k+1]_{p,q}}{p^{2(k+1)}}\\&& \hspace{.2 cm} +
  \frac{p^{n}}{p^\frac{n(n-1)}{2}}2\alpha [n+m]_{p,q} b_n^2\sum_{k=0}^{n+m-1} \left[\begin{array}{c} n+m-1 \\ k \end{array} \right]_{p,q} p^\frac{k(k+1)}{2}\left(\dfrac{x}{b_n}\right)^{k+1} \prod_{s=0}^{(n+m-k-2)} (p^s-q^s\dfrac{x}{b_n}) \dfrac{1}{p^{(k+1)}} +\alpha^2 b_n^2\Biggl]. 
\end{eqnarray*}
Now using $[k+1]_{p,q}=p^k+q[k]_{p,q}$, we will achieve the result.\\
Using linear property of operators, we have
\begin{eqnarray*}
C_{n,m}^{(\alpha,\beta)}((t-x);x,p,q)&=&C_{n,m}^{(\alpha,\beta)}(t;x,p,q)-xC_{n,m}^{(\alpha,\beta)}(1;x,p,q)\\
&=&\left(\dfrac{[n+m]_{p,q}}{[n]_{p,q}+\beta}-1\right)x+\dfrac{\alpha b_n}{[n]_{p,q}+\beta}.
\end{eqnarray*}
Hence, we get (iv).\\
Similar calculations give
$$
C_{n,m}^{(\alpha,\beta)}((t-x)^2;x,p,q)=C_{n,m}^{(\alpha,\beta)}(t^2;x,p,q)-2xC_{n,m}^{(\alpha,\beta)}(t;x,p,q)+x^2C_{n,m}^{(\alpha,\beta)}(1;x,p,q).$$
Substituting the resluts of (i),(ii) and (iii), we prove the result (v).\end{proof}
\begin{lemma} \label{L3}
For every fixed $p,q\in(0,1)$, we have
\begin{equation*}
\dfrac{[n+m]_{p,q}[n+m-1]_{p,q}}{([n]_{p,q}+\beta)^2}q-2\dfrac{[n+m]_{p,q}}{[n]_{p,q}+\beta}+1 \leq \left(\dfrac{(p^n+q^n)[m]_{p,q}-\beta}{[n]_{p,q}+\beta}\right)^2.
\end{equation*}
\end{lemma}
\begin{proof}
Since $[n+m]_{p,q}[n+m-1]_{p,q} \leq [n+m]_{p,q}^2,$ we get
\begin{eqnarray*}
&&\dfrac{[n+m]_{p,q}[n+m-1]_{p,q}}{([n]_{p,q}+\beta)^2}q-2\dfrac{[n+m]_{p,q}}{[n]_{p,q}+\beta}+1  \leq 
\left(\dfrac{[n+m]_{p,q}}{[n]_{p,q}+\beta}-1\right)^2 \\
&&\hspace{2 cm}= \dfrac{1}{([n]_{p,q}+\beta)^2}\left\{ \dfrac{p^{n+m}-q^{n+m}}{p-q}-\dfrac{p^n-q^n}{p-q}-\beta\right\}^2\\
&&\hspace{2 cm}= \dfrac{1}{([n]_{p,q}+\beta)^2}\left\{ \dfrac{(p^{n}+q^{n})(p^m-q^m)-q^n(1-p^m)-p^n(1-q^m)}{p-q}-\beta\right\}^2\\
&&\hspace{2 cm}= \dfrac{1}{([n]_{p,q}+\beta)^2}\left\{(p^{n}+q^{n})[m]_{p,q}-\beta - \dfrac{q^n(1-p^m)+p^n(1-q^m)}{p-q} \right\}^2 \\
&&\hspace{2 cm}\leq   \dfrac{((p^{n}+q^{n})[m]_{p,q}-\beta)^2}{([n]_{p,q}+\beta)^2} \qquad \mathrm{since},~ 0 < q < p \leq 1.
\end{eqnarray*} 
\end{proof}
\begin{remark} As a result of Lemma \ref{L2} and \ref{L3}, we have
\begin{eqnarray*}
&&C_{n,m}^{(\alpha,\beta)}((t-x)^2;x,p,q)\leq \left( \dfrac{((p^{n}+q^{n})[m]_{p,q}-\beta)^2}{([n]_{p,q}+\beta)^2} \right)x^2
+\left( \dfrac{[n+m]_{p,q}(2\alpha+p^{n-1})}{([n]_{p,q}+\beta)}-2 \alpha\right) \frac{b_nx}{([n]_{p,q}+\beta)}\\&&\hspace{5 cm}+\dfrac{\alpha^2 b_n^2}{([n]_{p,q}+\beta)^2}.
\end{eqnarray*}
\end{remark}
\begin{lemma} Taking supremum on $[0,b_n]$ on
Second central moment, we get
\begin{eqnarray*}
C_{n,m}^{(\alpha,\beta)}((t-x)^2;x,p,q) \leq
 b_n^2\Biggl\{ \dfrac{((p^{n}+q^{n})[m]_{p,q}-\beta)^2}{([n]_{p,q}+\beta)^2}+\Biggl|\dfrac{[n+m]_{p,q}(2\alpha+p^{n-1})}{([n]_{p,q}+\beta)^2}-\frac{2 \alpha}{([n]_{p,q}+\beta)}\Biggl|+ \dfrac{\alpha^2}{([n]_{p,q}+\beta)^2}\Biggl\}.
\end{eqnarray*}
\end{lemma}
\section{Korovkin-type approximation theorem}
Assume $C_\rho$ is the space of all continuous functions $f$ such that $$|f(x)| \leq M \rho(x),~~~~~~~~~ -\infty<x<\infty.$$ 
Then $C_\rho$ is a Banach space with the norm $$\|f\|_\rho=\sup_{-\infty<x<\infty}\dfrac{|f(x)|}{\rho(x)}.$$
The subsequent results are used for proving Korovkin approximation theorem on unbounded sets.
\begin{theorem} \label{TT1}(See \cite{korovUnbounded}) There exists a sequence of positive linear operators $U_n$, acting from $C_\rho$ to $C_\rho ,$ satisfying the conditions
\begin{enumerate}
\item[(1)] $\lim_{n\rightarrow \infty}\|U_n(1;x)-1\|_\rho=0,$
\item[(2)]$\lim_{n\rightarrow \infty}\|U_n(\phi;x)-\phi\|_\rho=0,$
\item[(3)]$\lim_{n\rightarrow \infty}\|U_n(\phi^2;x)-\phi^2\|_\rho=0,$
\end{enumerate}
where $\phi(x)$ is a continuous and increasing function on $(-\infty,\infty)$ such that $\lim_{x\rightarrow \pm \infty} \phi(x) =
\pm \infty$ and $\rho(x) = 1 + \phi^2$, and there exists a function $f^* \in C_\rho$ for which $\varlimsup_{n \to \infty} \|U_nf^* - f^*\|_\rho >0 $.\end{theorem}
\begin{theorem} \label{TT2} (See \cite{korovUnbounded})
Conditions (1),(2),(3) of above theorem implies that $$\lim_{n\to \infty}\|U_nf-f\|_\rho=0$$ for any function $f$ belonging to the subset $C_\rho^0 := \{ f \in C_\rho : lim_{|x| \to\infty} \frac{|f(x)|}{\rho(x)} ~is~ finite\}$ .
\end{theorem}
Consider the weight function $\rho(x)=1+x^2$ and operators:
$$U_{n,m}^{\alpha,\beta}(f;x,p,q)=\begin{cases}
C_{n,m}^{\alpha,\beta}(f;x,p,q)&\mathrm{if~} x \in [0,b_n],\\
f(x)& \mathrm{if~} x \in [0,\infty)/[0,b_n].\\
\end{cases}$$
 Thus for $f \in C_{1+x^2},$ we have
 \begin{align*}
 \|U_{n,m}^{\alpha,\beta}(f;\cdot ,p,q)\| &\leq \sup_{x\in[0,b_n]} \dfrac{|U_{n,m}^{\alpha,\beta}(f;x,p,q)|}{1+x^2}+\sup_{b_n<x<\infty}\dfrac{|f(x)|}{1+x^2} \\
 & \leq \|f\|_{1+x^2} \Biggl[ \sup_{x\in[0,\infty)} \dfrac{|U_{n,m}^{\alpha,\beta}(1+t^2;x,p,q)|}{1+x^2}+1\Biggl].
 \end{align*}

Now, using Lemma \ref{L2} we will obtain,
$$\|U_{n,m}^{\alpha,\beta}(f;\cdot ,p,q)\|_{1+x^2} \leq M \|f\|_{1+x^2}$$
if $0<q < p\leq 1,\lim_{n\to \infty}p^n=\lim_{n\to \infty}q^n=N, N<\infty ~and~ \lim_{n\to \infty}\frac{b_n}{[n]}=0.$
\begin{theorem} \label{TT3}
For all $f\in C_{1+x^2}^0$, we have $$\lim_{n\to \infty}\|U_{n,m}^{\alpha,\beta}(f;\cdot ,p_n,q_n)-f(\cdot)\|_{1+x^2}=0$$
provided that $p:=(p)_n,q:=(q)_n$ with $0<q_n < p_n\leq 1,\lim_{n\to \infty}p_n=1, \lim_{n\to \infty}q_n=1,\lim_{n\to \infty}p_n^n=\lim_{n\to \infty}q_n^n=N, N<\infty, lim_{n\to\infty}[n]_{p_n,q_n}=\infty ~and~ \lim_{n\to \infty}\frac{b_n}{[n]}=0.$
\end{theorem}
\begin{proof}
Using the results of Theorem \ref{TT1} and Lemma \ref{L2} (i),(ii) and (iii), we will achieve the following
assessments, respectively:
$$\sup_{x\in[0,\infty)}\dfrac{|U_{n,m}^{\alpha,\beta}(1;x,p_n,q_n)-1|}{1+x^2}=\sup_{ 0\leq x \leq b_n}\dfrac{|C_{n,m}^{\alpha,\beta}(1;x,p_n,q_n)-1|}{1+x^2}=0$$
$$\sup_{x \in [0,\infty)}\dfrac{|U_{n,m}^{\alpha,\beta}(t;x,p_n,q_n)-t|}{1+x^2}=\sup_{ 0\leq x \leq b_n}\dfrac{|C_{n,m}^{\alpha,\beta}(t;x,p_n,q_n)-x|}{1+x^2} \leq \sup_{0\leq x \leq b_n}\dfrac{\left|\dfrac{[n+m]_{p_n,q_n}}{[n]_{p_n,q_n}+\beta}-1\right|x+\dfrac{\alpha b_n}{[n]_{p_n,q_n}+\beta}}{1+x^2} $$$$
\leq \left|\dfrac{[n+m]_{p_n,q_n}}{[n]_{p_n,q_n}+\beta}-1\right|+\dfrac{\alpha b_n}{[n]_{p_n,q_n}+\beta} \rightarrow 0$$
and\\
\begin{eqnarray*}
&&\sup_{x\in[0,\infty)}\dfrac{|U_{n,m}^{\alpha,\beta}(t^2;x,p_n,q_n)-t^2|}{1+x^2}=\sup_{0\leq x \leq b_n}\dfrac{|C_{n,m}^{\alpha,\beta}(t^2;x,p_n,q_n)-x^2|}{1+x^2} 
\\&&\hspace{2 cm}\leq  \sup_{ 0\leq x \leq b_n}\dfrac{1}{1+x^2}\Biggl|\frac{(q_n[n+m]_{p_n,q_n}[n+m-1]_{p_n,q_n}x^2  +[n+m]_{p_n,q_n}(2\alpha+p_n^{n-1})b_n x+\alpha^2 b_n^2)}{([n]_{p_n,q_n}+\beta)^2}-x^2 \Biggl|\\
& & \hspace{2 cm}\leq  \Biggl\{\Biggl| \frac{(q_n[n+m]_{p_n,q_n}[n+m-1]_{p_n,q_n}}{([n]_{p_n,q_n}+\beta)^2}-1\Biggl|+\Biggl|\frac{[n+m]_{p_n,q_n}(2\alpha+p_n^{n-1})}{([n]_{p_n,q_n}+\beta)^2}\Biggl|+\frac{\alpha^2}{([n]_{p_n,q_n}+\beta)^2} \Biggl\} \longrightarrow   0
\end{eqnarray*}
whenever $n \to \infty$, because we have $\lim_{n\to \infty}p_n=\lim_{n\to \infty}q_n =1$ and $lim_{n\to\infty}[n]_{p_n,q_n}=\infty$ as $n \to \infty$.
\end{proof}
\begin{theorem}\label{TT4} Assuming $c$ as a positive and real number independent of $n$ and $f$ as a continuous
function which vanishes on $[c,\infty)$. Let $p := (p_n), q := (q_n)$ with $0 < q_n < p_n \leq 1, \lim_{n\to \infty} p_n =\lim_{n\to \infty} q_n= 1,~\lim_{n\to \infty} p_n^n = \lim_{n\to \infty} q_n^n= N < \infty $ and $\lim_{n\to \infty} \frac{b_n^2}{[n]_{p,q}} =0$. Then we have $$\lim_{n\to \infty}\sup_{0 \leq x\leq b_n} \left| C_{n,m}^{\alpha,\beta}(f;x,p_n,q_n)- f(x) \right|=0.$$ 
\end{theorem}
\begin{proof}
 From the hypothesis on $f$, it is bounded i.e. $|f(x)| \leq M ~(M>0)$. For any $\epsilon >0$, we have
$$\left|f\left(\dfrac{p_n^{n-k}[k]_{p_n,q_n}+\alpha}{[n]_{p_n,q_n}+\beta}b_n \right)-f(x)\right|<\epsilon+\dfrac{2M}{\delta^2}\left(\dfrac{p_n^{n-k}[k]_{p_n,q_n}+\alpha}{[n]_{p_n,q_n}+\beta}b_n-x\right)^2,$$
where $x\in[0,b_n]$ and $\delta=\delta(\epsilon)$ are independent of $n$. Operating with the operator \ref{1}  on both sides,
we can conclude by using Theorem \ref{TT3} and Lemma \ref{L3},
\begin{align*}
\sup_{0 \leq x\leq b_n} \left| C_{n,m}^{\alpha,\beta}(f;x,p_n,q_n)- f(x) \right| \leq \epsilon &+\dfrac{2M}{\delta^2} b_n^2\Biggl\{ \Biggl|1-2\frac{[n+m]_{p_n,q_n}}{([n]_{p_n,q_n}+\beta)}+\frac{q_n[n+m]_{p_n,q_n}[n+m-1]_{p_n,q_n}}{([n]_{p_n,q_n}+\beta)^2} \Biggl|\\  &+
   \Biggl| \frac{(2\alpha+p_n^{n-1})[n+m]_{p_n,q_n}}{([n]_{p_n,q_n}+\beta)}-2\alpha\Biggl| \dfrac{1}{([n]_{p_n,q_n}+\beta)}  +\dfrac{\alpha^2 }{([n]_{p_n,q_n}+\beta)^2} \Biggl\}.
\end{align*}
Since $\frac{b_n^2}{[n]_{p_n,q_n}} =0$ as $n \to \infty,$ we have the desired result.\end{proof}
\section{Rate of Convergence}
We will find the rate of convergence in terms of the Lipschitz class $Lip_M(\gamma )$, for $0 < \gamma \leq 1$. Assume that $C_B[0,\infty)$ denote the space of bounded continuous functions on $[0,\infty)$. A function $f \in C_B[0,\infty)$ belongs to $Lip_M(\gamma )$ if
$$ |f (t) - f (x)| \leq M|t - x|^\gamma,~~~~ t, x \in [0,\infty)$$
is satisfied.\\
\begin{theorem} \label{T1} Let $f \in Lip_M(\gamma )$, then
$$C_{n,m}^{\alpha,\beta}(f;x,p,q)\leq M(\lambda_{n,p,q}(x))^{\gamma /2},$$ where
$\lambda_{n,p,q}(x)=C_{n,m}^{\alpha,\beta}((t-x)^2;x,p,q).$\end{theorem}
\begin{proof}
Since $f \in Lip_M(\gamma )$, and the operator $C_{n,m}^{\alpha,\beta}(f;x,p,q)$ is linear and monotone,
\begin{eqnarray*}
&&| C_{n,m}^{\alpha,\beta}(f;x,p,q)- f(x)|
= \Biggl|\frac{1}{p^\frac{n(n-1)}{2}}\sum_{k=0}^{n+m} \left[\begin{array}{c} n+m \\ k \end{array} \right]_{p,q} p^\frac{k(k-1)}{2}\left(\dfrac{x}{b_n}\right)^k \prod_{s=0}^{(n+m-k-1)} (p^s-q^s\dfrac{x}{b_n}) \\&&\hspace{6 cm} f \left(\dfrac{p^{n-k}[k]_{p,q}+\alpha}{[n]_{p,q}+\beta}b_n\right)-f(x)\Biggl|\\
&&\hspace{1 cm}\leq \frac{1}{p^\frac{n(n-1)}{2}}\sum_{k=0}^{n+m} \left[\begin{array}{c} n+m \\ k \end{array} \right]_{p,q} p^\frac{k(k-1)}{2}\left(\dfrac{x}{b_n}\right)^k \prod_{s=0}^{(n+m-k-1)} (p^s-q^s\dfrac{x}{b_n}) \Biggl| f \left(\dfrac{p^{n-k}[k]_{p,q}+\alpha}{[n]_{p,q}+\beta}b_n\right)-f(x)\Biggl| \\
&&\hspace{1 cm} \leq  M\frac{1}{p^\frac{n(n-1)}{2}}\sum_{k=0}^{n+m} \left[\begin{array}{c} n+m \\ k \end{array} \right]_{p,q} p^\frac{k(k-1)}{2}\left(\dfrac{x}{b_n}\right)^k \prod_{s=0}^{(n+m-k-1)} (p^s-q^s\dfrac{x}{b_n}) \Biggl|  \left(\dfrac{p^{n-k}[k]_{p,q}+\alpha}{[n]_{p,q}+\beta}b_n\right)-x\Biggl|^\gamma.
\end{eqnarray*}
Using Remark (1) and H\"{o}lder's inequality with the values $p=\frac{2}{\gamma}$ and $q=\frac{2}{2-\gamma}$, we get
\begin{eqnarray*}
&&| C_{n,m}^{\alpha,\beta}(f;x,p,q)- f(x)|\\
&& \leq \frac{M}{p^\frac{n(n-1)}{2}} \sum_{k=0}^{n+m} \Biggl[\Biggl\{\left[\begin{array}{c} n+m \\ k \end{array} \right]_{p,q} p^\frac{k(k-1)}{2} \left(\dfrac{x}{b_n}\right)^k \prod_{s=0}^{(n+m-k-1)} (p^s-q^s\dfrac{x}{b_n}) \left( \left(\dfrac{p^{n-k}[k]_{p,q}+\alpha}{[n]_{p,q}+\beta}b_n\right) -x\right)^2\Biggl\}^{\frac{\gamma}{2}} \\& &\hspace{3 cm} \times \Biggl\{  \left[\begin{array}{c} n+m \\ k \end{array} \right]_{p,q}p^\frac{k(k-1)}{2} \left(\dfrac{x}{b_n}\right)^k \prod_{s=0}^{(n+m-k-1)} (p^s-q^s\dfrac{x}{b_n}) \Biggl\}^{\frac{2-\gamma}{2}} \Biggl]\\
&&\leq 
M  \Biggl[\Biggl\{\frac{1}{p^\frac{n(n-1)}{2}}\sum_{k=0}^{n+m}\left[\begin{array}{c} n+m \\ k \end{array} \right]_{p,q} p^\frac{k(k-1)}{2} \left(\dfrac{x}{b_n}\right)^k \prod_{s=0}^{(n+m-k-1)} (p^s-q^s\dfrac{x}{b_n}) \left( \left(\dfrac{p^{n-k}[k]_{p,q}+\alpha}{[n]_{p,q}+\beta}b_n\right) -x\right)^2\Biggl\}^{\frac{\gamma}{2}} \\& & \hspace{3 cm}\times \Biggl\{\frac{1}{p^\frac{n(n-1)}{2}} \sum_{k=0}^{n+m} \left[\begin{array}{c} n+m \\ k \end{array} \right]_{p,q}p^\frac{k(k-1)}{2} \left(\dfrac{x}{b_n}\right)^k \prod_{s=0}^{(n+m-k-1)} (p^s-q^s\dfrac{x}{b_n}) \Biggl\}^{\frac{2-\gamma}{2}} \Biggl]
\\
&=& M \left[C_{n,m}^{\alpha,\beta}((t-x)^2;x,p,q)\right]^{\frac{\gamma}{2}}\\
&\leq & M (\lambda_{n,p,q}(x))^{\frac{\gamma}{2}}.
\end{eqnarray*}
\end{proof} 
In order to obtain rate of convergence in terms of modulus of continuity $\omega(f;\delta)$, we assume that for any $f \in C_B[0,\infty)$ and $x \geq 0$, modulus of continuity of $f$ is given by
\begin{equation}
\omega(f;\delta)=\max_{\substack
{|t-x|\leq \delta \\ t,x\in [0,\infty)}}|f(t)-f(x)|.
\end{equation}
Thus it implies for any $\delta > 0$
\begin{equation}\label{E2}
|f(x)-f(y)| \leq \omega(f;\delta)\left( \dfrac{|x-y|}{\delta}+1\right).
\end{equation}
\begin{theorem}\label{T2} If $f \in C_B[0,\infty),$ we have 
$$|C_{n,m}^{\alpha,\beta}(f;x,p,q)- f(x)| \leq 2 \omega (f;\sqrt[]{\lambda_{n,p,q}(x)}),$$
where $\omega(f;\cdot)$ is modulus of continuity of $f$ and $\lambda_{n,p,q}(x)$ be the same as in Theorem \ref{T1}.
\end{theorem}
\begin{proof}
 Using triangular inequality, we get
\begin{eqnarray*}
&&|C_{n,m}^{\alpha,\beta}(f;x,p,q)- f(x)|= \Biggl|\frac{1}{p^\frac{n(n-1)}{2}}\sum_{k=0}^{n+m} \left[\begin{array}{c} n+m \\ k \end{array} \right]_{p,q} p^\frac{k(k-1)}{2}\left(\dfrac{x}{b_n}\right)^k \prod_{s=0}^{(n+m-k-1)} (p^s-q^s\dfrac{x}{b_n}) \\&&\hspace{6 cm} f \left(\dfrac{p^{n-k}[k]_{p,q}+\alpha}{[n]_{p,q}+\beta}b_n\right)-f(x)\Biggl|\\
&&\hspace{1 cm}  \leq  \frac{1}{p^\frac{n(n-1)}{2}}\sum_{k=0}^{n+m} \left[\begin{array}{c} n+m \\ k \end{array} \right]_{p,q} p^\frac{k(k-1)}{2}\left(\dfrac{x}{b_n}\right)^k \prod_{s=0}^{(n+m-k-1)} (p^s-q^s\dfrac{x}{b_n}) \Biggl| f \left(\dfrac{p^{n-k}[k]_{p,q}+\alpha}{[n]_{p,q}+\beta}b_n\right)-f(x)\Biggl| .
\end{eqnarray*}
Now using (\ref{E2}) and H\"{o}lder's inequality, we get
\begin{eqnarray*}
&&|C_{n,m}^{\alpha,\beta}(f;x,p,q)- f(x)|\\
&&\hspace{.5 cm}= \frac{1}{p^\frac{n(n-1)}{2}}\sum_{k=0}^{n+m} \left[\begin{array}{c} n+m \\ k \end{array} \right]_{p,q} p^\frac{k(k-1)}{2}\left(\dfrac{x}{b_n}\right)^k \prod_{s=0}^{(n+m-k-1)} (p^s-q^s\dfrac{x}{b_n}) \Biggl( \frac{ \left|\dfrac{p^{n-k}[k]_{p,q}+\alpha}{[n]_{p,q}+\beta}b_n-x\right|}{\delta}+1\Biggl)\omega(f;\delta)\\
&&\hspace{.5 cm} \leq \omega(f;\delta) \frac{1}{p^\frac{n(n-1)}{2}}\sum_{k=0}^{n+m} \left[\begin{array}{c} n+m \\ k \end{array} \right]_{p,q} p^\frac{k(k-1)}{2}\left(\dfrac{x}{b_n}\right)^k \prod_{s=0}^{(n+m-k-1)} (p^s-q^s\dfrac{x}{b_n})\\&&\hspace{1 cm}+\dfrac{\omega(f;\delta)}{\delta} \frac{1}{p^\frac{n(n-1)}{2}}\sum_{k=0}^{n+m} \left[\begin{array}{c} n+m \\ k \end{array} \right]_{p,q} p^\frac{k(k-1)}{2}\left(\dfrac{x}{b_n}\right)^k \prod_{s=0}^{(n+m-k-1)} (p^s-q^s\dfrac{x}{b_n}) \left|\frac{p^{n-k}[k]_{p,q}+\alpha}{[n]_{p,q}+\beta}b_n-x\right| \\
&&\hspace{.5 cm}= \omega(f;\delta)+\dfrac{\omega(f;\delta)}{\delta}\Biggl\{\frac{1}{p^\frac{n(n-1)}{2}}\sum_{k=0}^{n+m} \left[\begin{array}{c} n+m \\ k \end{array} \right]_{p,q} p^\frac{k(k-1)}{2}\left(\dfrac{x}{b_n}\right)^k \prod_{s=0}^{(n+m-k-1)} (p^s-q^s\dfrac{x}{b_n}) \\&& \hspace{5 cm}\left(\frac{p^{n-k}[k]_{p,q}+\alpha}{[n]_{p,q}+\beta}b_n-x\right)^2\Biggl\}^{\frac{1}{2}}\\
&&\hspace{.5 cm} = \omega(f;\delta)+\dfrac{\omega(f;\delta)}{\delta}\left\{C_{n,m}^{\alpha,\beta}((t-x)^2;x,p,q)\right\}^{1/2}.
\end{eqnarray*}
Now choosing $\delta=\lambda_{n,p,q}(x)$ as in Theorem \ref{T1}, we have
$$|C_{n,m}^{\alpha,\beta}(f;x,p,q)- f(x)| \leq 2 \omega(f;\sqrt{\lambda_{n,p,q}(x)}). $$
\end{proof}
Next we calculate the rate of convergence in terms of the modulus of continuity of the derivative of function.
\begin{theorem}\label{T3}
If $f(x)$ has a continuous bounded derivative $f ^{'}(x)$ and $\omega(f^{'} ; \delta)$ is the modulus
of continuity of $f^{'}(x)$ in $[0,A]$, then
\begin{eqnarray*}
&&|f(x)-C_{n,m}^{\alpha,\beta}(f;x,p,q)|\leq \\
&&~~~~M\Biggl(\left| \dfrac{[n+m]_{p,q}}{[n]_{p,q}+\beta}-1 \right|A+\dfrac{\alpha b_n}{[n]_{p,q}+\beta}\Biggl)+2(B_{n,p,q}(\alpha,\beta))^{1/2} \omega(f^{'};(B_{n,p,q}(\alpha,\beta))^{1/2}),
\end{eqnarray*}
where $M$ is a positive constant such that $|f^{'}(x)| \leq M$.
\end{theorem}
\begin{proof}
Using Mean Value Theorem, we have
\begin{eqnarray*}
f \left(\dfrac{p^{n-k}[k]_{p,q}+\alpha}{[n]_{p,q}+\beta}b_n\right)-f(x) &=&  \left(\dfrac{p^{n-k}[k]_{p,q}+\alpha}{[n]_{p,q}+\beta}b_n-x\right) f^{'}(\xi)
\\&=& \left(\dfrac{p^{n-k}[k]_{p,q}+\alpha}{[n]_{p,q}+\beta}b_n-x\right)f^{'}(x)+ \left(\dfrac{p^{n-k}[k]_{p,q}+\alpha}{[n]_{p,q}+\beta}b_n-x\right)(f^{'}(\xi)-f^{'}(x)),
\end{eqnarray*}
where $\xi$ is a point between $x$ and $\frac{p^{n-k}[k]_{p,q}+\alpha}{[n]_{p,q}+\beta}b_n$. By using the above identity, we get
\begin{eqnarray*}
&&C_{n,m}^{\alpha,\beta}(f;x,p,q)-f(x) \\ &&~~~~
=f^{'}(x) \frac{1}{p^\frac{n(n-1)}{2}}\sum_{k=0}^{n+m} \left[\begin{array}{c} n+m \\ k \end{array} \right]_{p,q} p^\frac{k(k-1)}{2}\left(\dfrac{x}{b_n}\right)^k \prod_{s=0}^{(n+m-k-1)} (p^s-q^s\dfrac{x}{b_n}) \left(\frac{p^{n-k}[k]_{p,q}+\alpha}{[n]_{p,q}+\beta}b_n-x\right)\\&& ~~~~~~+\frac{1}{p^\frac{n(n-1)}{2}}\sum_{k=0}^{n+m} \left[\begin{array}{c} n+m \\ k \end{array} \right]_{p,q} p^\frac{k(k-1)}{2}\left(\dfrac{x}{b_n}\right)^k \prod_{s=0}^{(n+m-k-1)} (p^s-q^s\dfrac{x}{b_n}) \left(\frac{p^{n-k}[k]_{p,q}+\alpha}{[n]_{p,q}+\beta}b_n-x\right)(f^{'}(\xi)-f^{'}(x)).
\end{eqnarray*}
Hence,
\begin{eqnarray*}
&&|C_{n,m}^{\alpha,\beta}(f;x,p,q)-f(x)| \\ &&~~ \leq |f^{'}(x)| |C_{n,m}^{\alpha,\beta}((t-x);x,p,q)| \\ && ~~~~ + 
\frac{1}{p^\frac{n(n-1)}{2}}\sum_{k=0}^{n+m} \left[\begin{array}{c} n+m \\ k \end{array} \right]_{p,q} p^\frac{k(k-1)}{2}\left(\dfrac{x}{b_n}\right)^k \prod_{s=0}^{(n+m-k-1)} (p^s-q^s\dfrac{x}{b_n}) \left|\frac{p^{n-k}[k]_{p,q}+\alpha}{[n]_{p,q}+\beta}b_n-x\right| |f^{'}(\xi)-f^{'}(x)| \\&& ~~ \leq M \Biggl( \left|\dfrac{[n+m]_{p,q}}{[n]_{p,q}+\beta}-1\right| A+\dfrac{\alpha b_n}{[n]_{p,q}+\beta}\Biggl) \\
&&~~~~ + 
\frac{1}{p^\frac{n(n-1)}{2}}\sum_{k=0}^{n+m} \left[\begin{array}{c} n+m \\ k \end{array} \right]_{p,q} p^\frac{k(k-1)}{2}\left(\dfrac{x}{b_n}\right)^k \prod_{s=0}^{(n+m-k-1)} (p^s-q^s\dfrac{x}{b_n}) \left|\frac{p^{n-k}[k]_{p,q}+\alpha}{[n]_{p,q}+\beta}b_n-x\right| |f^{'}(\xi)-f^{'}(x)| \\&& ~~ \leq 
M \Biggl( \left|\dfrac{[n+m]_{p,q}}{[n]_{p,q}+\beta}-1\right| A+\dfrac{\alpha b_n}{[n]_{p,q}+\beta}\Biggl) \\
&&~~~~ + 
\frac{1}{p^\frac{n(n-1)}{2}}\sum_{k=0}^{n+m} \left[\begin{array}{c} n+m \\ k \end{array} \right]_{p,q} p^\frac{k(k-1)}{2}\left(\dfrac{x}{b_n}\right)^k \prod_{s=0}^{(n+m-k-1)} (p^s-q^s\dfrac{x}{b_n}) \left|\frac{p^{n-k}[k]_{p,q}+\alpha}{[n]_{p,q}+\beta}b_n-x\right|\\
&&\hspace{4 cm} \times \omega(f^{'};\delta)\Biggl( \dfrac{\left|\frac{p^{n-k}[k]_{p,q}+\alpha}{[n]_{p,q}+\beta}b_n-x\right|}{\delta}+1 \Biggl),
\end{eqnarray*}
since \begin{equation*}
|\xi-x| \leq \left|\frac{p^{n-k}[k]_{p,q}+\alpha}{[n]_{p,q}+\beta}b_n-x\right|.
\end{equation*}
Using it, we have 
\begin{eqnarray*}
&&|C_{n,m}^{\alpha,\beta}(f;x,p,q)-f(x)| \\ &&~~ \leq M \Biggl( \left|\dfrac{[n+m]_{p,q}}{[n]_{p,q}+\beta}-1\right| A+\dfrac{\alpha b_n}{[n]_{p,q}+\beta}\Biggl) \\
&&~~~~ + 
\omega(f^{'};\delta)\frac{1}{p^\frac{n(n-1)}{2}}\sum_{k=0}^{n+m} \left[\begin{array}{c} n+m \\ k \end{array} \right]_{p,q} p^\frac{k(k-1)}{2}\left(\dfrac{x}{b_n}\right)^k \prod_{s=0}^{(n+m-k-1)} (p^s-q^s\dfrac{x}{b_n}) \left|\frac{p^{n-k}[k]_{p,q}+\alpha}{[n]_{p,q}+\beta}b_n-x\right| \\
&&~~~~ +
\dfrac{\omega(f^{'};\delta)}{\delta}\frac{1}{p^\frac{n(n-1)}{2}}\sum_{k=0}^{n+m} \left[\begin{array}{c} n+m \\ k \end{array} \right]_{p,q} p^\frac{k(k-1)}{2}\left(\dfrac{x}{b_n}\right)^k \prod_{s=0}^{(n+m-k-1)} (p^s-q^s\dfrac{x}{b_n}) \left|\frac{p^{n-k}[k]_{p,q}+\alpha}{[n]_{p,q}+\beta}b_n-x\right|^2.
\end{eqnarray*}
Now using Cauchy-Schwarz inequality for the second term, we obtain
\begin{eqnarray*}
&&|C_{n,m}^{\alpha,\beta}(f;x,p,q)-f(x)| \\ &&~~ \leq M \Biggl( \left|\dfrac{[n+m]_{p,q}}{[n]_{p,q}+\beta}-1\right| A+\dfrac{\alpha b_n}{[n]_{p,q}+\beta}\Biggl) \\
&&~~~~ + 
\omega(f^{'};\delta)\Biggl(\frac{1}{p^\frac{n(n-1)}{2}}\sum_{k=0}^{n+m} \left[\begin{array}{c} n+m \\ k \end{array} \right]_{p,q} p^\frac{k(k-1)}{2}\left(\dfrac{x}{b_n}\right)^k \prod_{s=0}^{(n+m-k-1)} (p^s-q^s\dfrac{x}{b_n}) \left|\frac{p^{n-k}[k]_{p,q}+\alpha}{[n]_{p,q}+\beta}b_n-x\right|^2 \Biggl)^{1/2} \\
&&~~~~ +
\dfrac{\omega(f^{'};\delta)}{\delta}\frac{1}{p^\frac{n(n-1)}{2}}\sum_{k=0}^{n+m} \left[\begin{array}{c} n+m \\ k \end{array} \right]_{p,q} p^\frac{k(k-1)}{2}\left(\dfrac{x}{b_n}\right)^k \prod_{s=0}^{(n+m-k-1)} (p^s-q^s\dfrac{x}{b_n}) \left|\frac{p^{n-k}[k]_{p,q}+\alpha}{[n]_{p,q}+\beta}b_n-x\right|^2.\\
&&~~ =M \Biggl( \left|\dfrac{[n+m]_{p,q}}{[n]_{p,q}+\beta}-1\right| A+\dfrac{\alpha b_n}{[n]_{p,q}+\beta}\Biggl) \\
&&~~~~ + 
\omega(f^{'};\delta)\sqrt{C_{n,m}^{\alpha,\beta}((t-x)^2;x,p,q)}+\dfrac{\omega(f^{'};\delta)}{\delta}C_{n,m}^{\alpha,\beta}((t-x)^2;x,p,q).
\end{eqnarray*}
Using Remark (1), we see
\begin{eqnarray*}
&& \sup_{0\leq x \leq A}   C_{n,m}^{(\alpha,\beta)}((t-x)^2;x,p,q)\\&&~~ \leq \sup_{0\leq x \leq A}  \left( \dfrac{((p^{n}+q^{n})[m]_{p,q}-\beta)^2}{([n]_{p,q}+\beta)^2} \right)x^2
+\left( \dfrac{[n+m]_{p,q}(2\alpha+p^{n-1})}{([n]_{p,q}+\beta)}-2 \alpha\right) \frac{b_nx}{([n]_{p,q}+\beta)}\\&&\hspace{5 cm}+\dfrac{\alpha^2 b_n^2}{([n]_{p,q}+\beta)^2} \\&& \leq
\left( \dfrac{((p^{n}+q^{n})[m]_{p,q}-\beta)^2}{([n]_{p,q}+\beta)^2} \right)A^2
+\left( \dfrac{[n+m]_{p,q}(2\alpha+p^{n-1})}{([n]_{p,q}+\beta)}-2 \alpha\right) \frac{b_nA}{([n]_{p,q}+\beta)}\\&&\hspace{5 cm}+\dfrac{\alpha^2 b_n^2}{([n]_{p,q}+\beta)^2} \\&& ~~ :=B_{n,p,q}(\alpha,\beta).
\end{eqnarray*}
Thus,
\begin{eqnarray*}
|C_{n,m}^{\alpha,\beta}(f;x,p,q)-f(x)|  \leq M \Biggl( \left|\dfrac{[n+m]_{p,q}}{[n]_{p,q}+\beta}-1\right| A+\dfrac{\alpha b_n}{[n]_{p,q}+\beta}\Biggl)+ 
\omega(f^{'};\delta)\Biggl[ (B_{n,p,q}(\alpha,\beta))^{1/2}+\dfrac{1}{\delta}B_{n,p,q}(\alpha,\beta)\Biggl].
\end{eqnarray*}
Choosing, $\delta :=(B_{n,p,q}(\alpha,\beta))^{1/2} $, we get the desired result.
\end{proof}
\begin{flushleft}
\textbf{Example:}
\end{flushleft}
\textbf{Approximation of $f(x)=x^2$}\\ Approximation of $x^2 \in C_{1+x^2}[0,\infty)$ function within $[0,0.1]$,  obtained for $n=10$, 8 with $m=1$ and $b_n=(log~ n)^2$ is represented by the following figure.
\begin{figure}[h!]
\begin{center}
\includegraphics[scale=0.8]{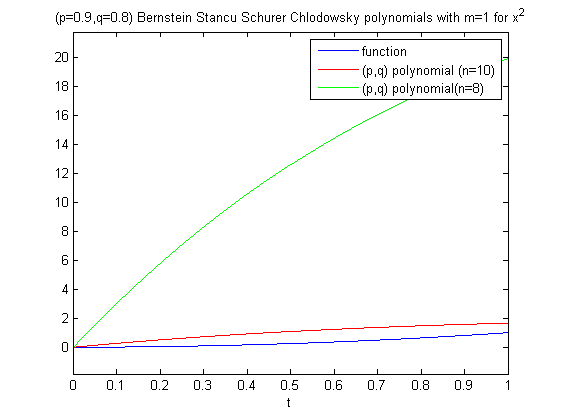}
\end{center}
\end{figure}
\pagebreak

It is clear from the figure that approximation is readily obtained for lower values of $n$ using $(p,q)$ variant of Bernstein-Stancu-Schurer-Chlodowsky polynomials in small dimension as it is required in quantum physics.
\begin{flushleft}
\textbf{Conflict of Interest:} The authors declare that they have no competing interests regarding the publication of this manuscript.
\end{flushleft}
\begin{flushleft}
\textbf{References}
\end{flushleft}

\end{document}